\documentclass[leqno]{sl2art}

\makeatletter
\def\qed@warning{}
\makeatother

\usepackage{mathtools}

\usepackage{graphicx,subcaption}
\graphicspath{{fig/}}

\usepackage{todonotes}

\usepackage{enumitem}
\newlist{thmenum}{enumerate}{1}
\setlist[thmenum]{label=(\alph*)}

\RequirePackage[noabbrev,capitalize]{cleveref}
\crefname{equation}{equation}{equations}
\Crefname{equation}{Equation}{Equations}

\newcommand{\set}[1]{\left\{ \def\given{\ \middle| \ }  #1 \right\}  }
\newcommand{\defeq}{\mathrel{:=}}

\DeclareMathOperator{\tr}{tr}

\newcommand{\ii}{\mathsf{i}}

\newcommand{\transp}{\mathrm{T}}

\def\leftfun#1{\mathopen{}\left#1}
\def\rightfun#1{\right#1}

\NewDocumentCommand{\slg}{}{\operatorname{SL}_2(\mathbb{C})}
\NewDocumentCommand{\pslg}{}{\operatorname{PSL}_2(\mathbb{C})}
\NewDocumentCommand{\glg}{}{\operatorname{GL}_2(\mathbb{C})}
\NewDocumentCommand{\psp}{O{\mathbb{C}} O{1}}{{#1 P}^{#2}}

\NewDocumentCommand{\upf}{m}{\varphi^+\leftfun(#1\rightfun)}
\NewDocumentCommand{\dnf}{m}{\varphi^-\leftfun(#1\rightfun)}
\NewDocumentCommand{\aroundhol}{m}{\tau\leftfun(#1\rightfun)}

\NewDocumentCommand{\upr}{m}{#1^{\uparrow}}
\NewDocumentCommand{\dnr}{m}{#1^{\downarrow}}

\NewDocumentCommand{\gpoid}{m}{\Pi\leftfun(#1\rightfun)}

\NewDocumentCommand{\Hol}{m}{\operatorname{Hol}_{#1}}
\NewDocumentCommand{\hol}{m}{\operatorname{hol}_{#1}}

\NewDocumentCommand{\evec}{m}{\mathrm{\mathbf{e}}_{#1}}

\DeclareMathOperator{\dil}{Li_2}
\DeclareMathOperator{\logdil}{L}

\ExplSyntaxOn
\DeclareExpandableDocumentCommand{\IfEmptyTF}{mmm}
 {
    \tl_if_empty:nTF {#1} {#2} {#3}
 }
\ExplSyntaxOff

\NewDocumentCommand{\potl}{o m}{
  \Phi_{\IfNoValueTF{#1}{}{#1}}%
  \IfEmptyTF{#2}{%
  }{%
    \leftfun(#2\rightfun)
  }
}

\declaretheorem[style=bigtheorem,title=Theorem,refname={Theorem, Theorems}]{bigtheorem}
\declaretheorem[style=theorem]{proposition}
\declaretheorem[style=theorem,sibling=proposition]{theorem}
\declaretheorem[style=theorem,sibling=proposition]{lemma}
\declaretheorem[style=theorem,sibling=proposition]{conjecture}
\declaretheorem[style=theorem,sibling=proposition]{corollary}
\declaretheorem[style=definition,sibling=proposition]{definition}

\declaretheorem[style=definition,sibling=proposition]{example}

\usepackage{biblatex}
\title{Octahedral coordinates from the Wirtinger presentation}
\author{Calvin McPhail-Snyder}
\address{Duke University}
\email{calvin@sl2.site}


\begin{document}

\begin{abstract}
  Let \(\rho\) be a representation of a knot group (or more generally, the fundamental group of a tangle complement) into \(\slg\) expressed in terms of the Wirtinger generators of a diagram \(D\).
  This diagram also determines an ideal triangulation of the complement called the octahedral decomposition.
  \(\rho\) induces a hyperbolic structure on the complement of \(D\), and in this note we give a direct algebraic formula for the geometric parameters of the octahedral decomposition induced by this structure.
  Our formula gives a new, explicit criterion for whether \(\rho\) occurs as a critical point of the diagram's Neumann-Zagier--Yokota potential function.
\end{abstract}

\maketitle

\tableofcontents

\section{Introduction}

Studying representations of knot and \(3\)-manifold groups into \(\slg\) is important in low-dimensional topology.
The isometry group of hyperbolic \(3\)-space is \(\pslg\), so this problem is closely related to the study of hyperbolic structures \cite{ThurstonNotes}.
While one can study representations algebraically, say via the Wirtinger presentation of the knot group, it is sometimes better to work more geometrically.
For example, for geometric invariants like the hyperbolic volume it is usually better to work with the geometric parameters of an ideal triangulation of the knot complement than directly with \(\rho\).

In this paper we are concerned with a particular family of ideal triangulations adapted to knot diagrams.
Given a diagram \(D\) of a knot \(K\) the \defemph{octahedral decomposition} decomposes \(S^3 \setminus K\) minus two points into a union of octahedra, with one for each crossing of the diagram.
By further subdividing into tetrahedra one can obtain an ideal triangulation of a standard form.
This works just as well for links and tangles.
The geometric data of the tetrahedra are naturally described by \defemph{octahedral coordinates}; these are implicit in work of \citeauthor{Yokota2000} \cite{Yokota2000} and have been studied systematically by \citeauthor{Kim2016} \cite{Kim2016,Kim2018}.
When studying geometric properties of a representation the octahedral coordinates are quite useful; for example, they enable a direct computation of the hyperbolic volume and Chern-Simons invariant \cite{Cho2013}.
Our version of these coordinates is motivated by connections to the representation theory of quantum groups \cite{McPhailSnyder2022,McPhailSnyderAlgebra}.

While the octahedral coordinates are geometrically natural they can be difficult to solve for in practice.
Instead we might find representations of the knot group by some other method and then try to find the corresponding octahedral coordinates.
Given a representation \(\rho : \pi_1(S^3 \setminus K) \to \slg\) described by its values on Wirtinger generators one can use the methods of \citeauthor{Blanchet2018} \cite{Blanchet2018} to find octahedral coordinates corresponding to \(\rho\).
This gives a somewhat complicated inductive method that does not usually produce simple formulas.
In addition, some representations do not come from octahedral coordinates (one might have to conjugate first) and there is not a simple way to check when this occurs.

In this paper we give an elementary method (\cref{def:associated-octahedral,thm:holonomies-match}) for determining the octahedral coordinates of a tangle diagram directly from \(\rho\).
As an application we give an explicit criterion (\cref{thm:detection-condition}) for whether a representation corresponds to a smooth critical point of the potential function of a knot diagram \(D\);
this question is relevant to the saddle-point method used in most approaches to the Volume Conjecture.
Our characterization depends only on \(\rho\) and \(D\), not the details of the triangulation.

The Volume Conjecture proposes that the asymptotics of the Kashaev invariants of a hyperbolic knot recover geometric information like the hyperbolic volume \cite{Kashaev1997}.
The Kashaev invariants are naturally associated with the octahedral decomposition, and most proofs of special cases of the conjecture proceed via studying an associated \defemph{potential function} \cite{Yokota2000}.
The potential function is defined on a space parametrizing the geometry of the ideal tetrahedra: it has singularities where they are degenerate and critical points at nondegenerate solutions of the gluing equations of the triangulation.
To apply the saddle-point method to the asymptotics of the Kashaev invariant it is important to use a diagram whose potential function has a smooth critical point at the complete hyperbolic structure.
This question has been previously studied using other techniques \cite{Sakuma2018,Garoufalidis2016} and it is known that every hyperbolic knot that is alternating or has at most 12 crossings has such a diagram.
We show in \cref{thm:arc-faithful} that for knot diagrams this condition is equivalent to being  \defemph{arc-faithful}, meaning that the Wirtinger generators of the over and under arcs at each crossing are always distinct.
This new characterization may prove useful for showing that the potential function has smooth critical points.

Quantum topology provides another motivation for our results.
One can define \cite{McPhailSnyderThesis,McPhailSnyderVolume} a geometrically twisted version of the Kashaev invariant that can also be understood as a quantization of the \(\slg\) Chern-Simons invariant, which includes the hyperbolic volume.
The construction of this invariant uses the octahedral coordinates in an essential way and we expect the results of this paper to be useful for computing and studying it.

This paper began as an attempt to re-derive and generalize a formula \cite[Theorem 3.2]{Cho2014} of \citeauthor{Cho2014} for boundary-parabolic representations.
\Cref{eq:chi-formula} is similar to but more general than \cite[Theorem 3.2]{Cho2014} and our proof is quite different.
\citeauthor{Cho2014}'s result is more closely connected to the quandle \(\mathcal{P}\) of parabolic elements of \(\pslg\), and in particular to a presentation of it due to \citeauthor{Inoue2014} \cite{Inoue2014}.
It would be interesting to better understand how our results relate to \(\mathcal{P}\) and if this can be generalized to the boundary non-parabolic case.

\subsection*{Organization}
\begin{description}
  \item[\cref{sec:octahedral-colorings}] We define octahedral colorings of tangle diagrams and state how to derive them from a representation of the fundamental group of the tangle complement given in terms of the Wirtinger presentation.
  \item[\cref{sec:holonomy}] We describe the geometrically-derived holonomy of an octahedral coloring and use it to justify our definition of the octahedral coloring associated to a representation.
  \item[\cref{sec:gauge-transformations}] We explain how conjugating a representation \(\rho\) acts on the induced octahedral coloring and use this to show that every representation is conjugate to one coming from a coloring.
  \item[\cref{sec:potential}] We give some applications of our results to the Neumann-Zagier--Yokota potential function and the Volume Conjecture.
\end{description}

\subsection*{Acknowledgements}
I would like to thank Andy Putman for a comment on my MathOverflow question \cite{MathoverflowMeridians} about knot groups that led me to the proof of \cref{thm:arc-faithful}.
I would also like to thank Seonhwa Kim for some useful comments on an early draft of the paper and for introducing me to relevant work of \textcite{Sakuma2018} and \textcite{Garoufalidis2016}.
Finally, I thank the anonymous referee for helpful comments that lead to an improved exposition.

\subsection*{Conventions}
All tangles are smooth and oriented.
Our convention is that tangle diagrams go from left to right, so composition of tangles is horizontal and disjoint union is vertical.
Our sign conventions the and typical labeling of the parts of a crossing are given in \cref{fig:crossing-types}.
Composition of paths is also read left to right: \(fg\) means follow \(f\) then \(g\).
For this reason we usually assume an eigenvector is a row vector, not a column vector.
However, we write \(\evec{1}, \evec{2}\) for the standard \emph{column} vector basis of \(\mathbb{C}^{2}\).
For (column) vectors \(u_1, u_2\) we write
\(
  \det\left( u_1, u_2 \right)
\)
for the determinant of the matrix with columns \(u_1, u_2\), i.e.\@ for the standard antisymmetric pairing on \(\mathbb{C}^2\).

\begin{figure}
  \centering
\begingroup%
  \makeatletter%
  \providecommand\color[2][]{%
    \errmessage{(Inkscape) Color is used for the text in Inkscape, but the package 'color.sty' is not loaded}%
    \renewcommand\color[2][]{}%
  }%
  \providecommand\transparent[1]{%
    \errmessage{(Inkscape) Transparency is used (non-zero) for the text in Inkscape, but the package 'transparent.sty' is not loaded}%
    \renewcommand\transparent[1]{}%
  }%
  \providecommand\rotatebox[2]{#2}%
  \newcommand*\fsize{\dimexpr\f@size pt\relax}%
  \newcommand*\lineheight[1]{\fontsize{\fsize}{#1\fsize}\selectfont}%
  \ifx\svgwidth\undefined%
    \setlength{\unitlength}{283.55544662bp}%
    \ifx\svgscale\undefined%
      \relax%
    \else%
      \setlength{\unitlength}{\unitlength * \real{\svgscale}}%
    \fi%
  \else%
    \setlength{\unitlength}{\svgwidth}%
  \fi%
  \global\let\svgwidth\undefined%
  \global\let\svgscale\undefined%
  \makeatother%
  \begin{picture}(1,0.27643014)%
    \lineheight{1}%
    \setlength\tabcolsep{0pt}%
    \put(0,0){\includegraphics[width=\unitlength,page=1]{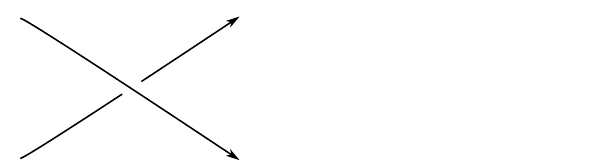}}%
    \put(-0.00351287,0.24426852){\makebox(0,0)[lt]{\lineheight{1.25}\smash{\begin{tabular}[t]{l}$1$\end{tabular}}}}%
    \put(-0.00351287,0.00621985){\makebox(0,0)[lt]{\lineheight{1.25}\smash{\begin{tabular}[t]{l}$2$\end{tabular}}}}%
    \put(0.41174984,0.24426852){\makebox(0,0)[lt]{\lineheight{1.25}\smash{\begin{tabular}[t]{l}$2'$\end{tabular}}}}%
    \put(0.41174984,0.00621985){\makebox(0,0)[lt]{\lineheight{1.25}\smash{\begin{tabular}[t]{l}$1'$\end{tabular}}}}%
    \put(0,0){\includegraphics[width=\unitlength,page=2]{crossing-types.pdf}}%
    \put(0.56563205,0.24427706){\makebox(0,0)[lt]{\lineheight{1.25}\smash{\begin{tabular}[t]{l}$1$\end{tabular}}}}%
    \put(0.56563205,0.00622839){\makebox(0,0)[lt]{\lineheight{1.25}\smash{\begin{tabular}[t]{l}$2$\end{tabular}}}}%
    \put(0.9808948,0.24427706){\makebox(0,0)[lt]{\lineheight{1.25}\smash{\begin{tabular}[t]{l}$2'$\end{tabular}}}}%
    \put(0.9808948,0.00622839){\makebox(0,0)[lt]{\lineheight{1.25}\smash{\begin{tabular}[t]{l}$1'$\end{tabular}}}}%
  \end{picture}%
\endgroup%

  \caption{Positive (left) and negative (right) crossings with our standard labeling.}
  \label{fig:crossing-types}
\end{figure}

\section{Octahedral colorings}
\label{sec:octahedral-colorings}

\begin{definition}
  The \defemph{complement} of a tangle \(T\) is its complement as a submanifold of \([0,1]^{3}\).
  We write \(\pi(T)\) for the fundamental group of the complement.
  Suppose \(T\) has \(n\) incoming and \(m\) outgoing boundary points and let \(D\) be a diagram of \(T\).
  We can think of \(D\) as a decorated \(4\)-valent graph \(G\) embedded in \([0,1] \times [0,1]\) with \(n\) edges intersecting \([0,1] \times \set{0}\) and \(m\) edges intersecting \([0,1] \times \set{1}\).
  We assign names to various parts of \(D\):
  \begin{itemize}
    \item
     The \defemph{segments} of \(D\) are the edges%
      \note{%
        Usually these are called the ``edges'' of the diagram, but we do not want to confuse them with edges of ideal polyhedra in the octahedral decomposition.
      }
      of \(G\).
    \item
    An \defemph{arc} of \(D\) is a set of adjacent over-segments.
    \item
      A \defemph{component} of \(D\) is a set of adjacent segments; these are in bijection with connected components of \(T\).
    \item
    A \defemph{region} of \(D\) is a connected component of the complement of \(G\), equivalently a vertex of the dual graph of \(G\).
    Regions are adjacent across segments, and our convention is to say that region \(j'\) is \defemph{below} region \(j\) when they are arranged as in \cref{fig:region-orientation-rule}.
    In this case we call \(j\) the region \defemph{above} \(i\) and write \(j = \upr i\).
    Similarly \(j'\) is the region \defemph{below} \(i\), written \(j = \dnr i\).
  \end{itemize}
\end{definition}

\begin{marginfigure}
\begingroup%
  \makeatletter%
  \providecommand\color[2][]{%
    \errmessage{(Inkscape) Color is used for the text in Inkscape, but the package 'color.sty' is not loaded}%
    \renewcommand\color[2][]{}%
  }%
  \providecommand\transparent[1]{%
    \errmessage{(Inkscape) Transparency is used (non-zero) for the text in Inkscape, but the package 'transparent.sty' is not loaded}%
    \renewcommand\transparent[1]{}%
  }%
  \providecommand\rotatebox[2]{#2}%
  \newcommand*\fsize{\dimexpr\f@size pt\relax}%
  \newcommand*\lineheight[1]{\fontsize{\fsize}{#1\fsize}\selectfont}%
  \ifx\svgwidth\undefined%
    \setlength{\unitlength}{109.92816067bp}%
    \ifx\svgscale\undefined%
      \relax%
    \else%
      \setlength{\unitlength}{\unitlength * \real{\svgscale}}%
    \fi%
  \else%
    \setlength{\unitlength}{\svgwidth}%
  \fi%
  \global\let\svgwidth\undefined%
  \global\let\svgscale\undefined%
  \makeatother%
  \begin{picture}(1,0.3133619)%
    \lineheight{1}%
    \setlength\tabcolsep{0pt}%
    \put(0,0){\includegraphics[width=\unitlength,page=1]{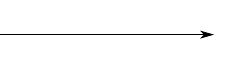}}%
    \put(0.341132,0.23042404){\makebox(0,0)[lt]{\lineheight{1.25}\smash{\begin{tabular}[t]{l}$j = \upr{i}$\end{tabular}}}}%
    \put(0.30019606,0.02574484){\makebox(0,0)[lt]{\lineheight{1.25}\smash{\begin{tabular}[t]{l}$j' = \dnr{i}$\end{tabular}}}}%
    \put(0.95516958,0.16219765){\makebox(0,0)[lt]{\lineheight{1.25}\smash{\begin{tabular}[t]{l}$i$\end{tabular}}}}%
  \end{picture}%
\endgroup%

  \caption{
    In this case we say that region \(j'\) is below region \(j\) across segment \(i\).
    We write \(j = \upr i\) and \(j' = \dnr i\) to refer to the regions adjacent to segment \(i\).
  }
  \label{fig:region-orientation-rule}
\end{marginfigure}

\begin{proposition}
  \label{thm:wirtinger}
  Let \(D\) be a diagram of a tangle \(T\).
  The choice of \(D\) gives a presentation for \(\pi(T)\) with one generator \(w_i\) for each arc and one relation
  \begin{gather}
    \label{eq:wirtinger-positive}
    w_{2'} = w_{1}^{-1} w_{2} w_{1} \text{ (positive)}
    \\
    \label{eq:wirtinger-negative}
    w_{1'} = w_{2} w_{1} w_{2}^{-1} \text{ (negative)}
  \end{gather}
  at each crossing, where the indices are as in \cref{fig:crossing-types}.
  (At a positive crossing \(1\) and \(1'\)  are the same arc.)
  We call this the \defemph{Wirtinger presentation} of \(\pi(T)\) and denote it by \(\pi(D)\).
  If \(i\) is a segment of \(D\) we write \(w_{i}\) for the associated Wirtinger generator.
\end{proposition}

\begin{definition}
  Let \(D\) be a diagram of a tangle \(T\) and \(G\) be a group.
  A \defemph{\(G\)-coloring} of \(D\) is a function \(i \mapsto g_i\) assigning each segment an element of \(g\), subject to the Wirtinger relations \eqref{eq:wirtinger-positive} and \eqref{eq:wirtinger-negative} at each crossing.%
  \note{
    Such a coloring will necessarily assign every segment in an arc the same group element, but we will soon consider octahedral colorings that do not have this property.
  }
\end{definition}

By \cref{thm:wirtinger} \(G\)-colorings of \(D\) are in bijection with representations \(\rho : \pi(D) \to G\).
It turns out that octahedral coordinates correspond to slightly more data: they describe both a representation of \(\pi(T)\) and some extra data on the boundary of a regular neighborhood of \(T\).

\begin{figure}
  \centering
  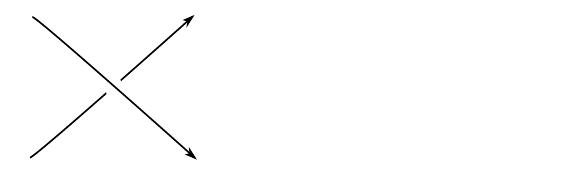
  \caption{Rules for decorated \(\slg\) colorings.}
  \label{fig:decorated-quandle-rule}
\end{figure}

\begin{definition}
  A \defemph{decorated representation} \(\rho\) of a tangle \(T\) is a representation \(\rho : \pi(T) \to \slg\) with a choice of invariant line for each meridian of \(T\).

  More formally, each component \(i\) of the oriented tangle has a conjugacy class of meridians \([w_i] \subset \pi(T)\).
  For a representative \(w_i\) we choose a line \(L_i \subset \mathbb{C}^2\) (thought of as a set of \emph{row} vectors) with
  \[
    L_i \rho(w_i) = L_i.
  \]
  This does not depend on the choice of representative: if \(w_i' = y^{-1} w_i y\) is any other representative of the conjugacy class we assign it the line \(L_i \rho(y)\), since
  \[
    L_i \rho(y) \rho(y^{-1} w_i y) = L_i \rho(y).
  \]
  This choice is called a \defemph{decoration} of component \(i\), and a decoration of \(\rho\) is a decoration of each of its components.
  A decoration of \(\rho\) induces an equivalent decoration of any conjugate \(g^{-1} \rho g\) in a similar way.
\end{definition}

Generically a knot has two decorations, as a diagonalizable matrix has two eigenspaces.
When \(\tr \rho(w) = \pm 2\) but \(\rho \ne \pm 1\) is nontrivial (i.e. when \(\rho\) is boundary-parabolic) there is only one decoration, and when \(\rho(w) = \pm 1\) is trivial there are infinitely many.
Similar statements apply to links and tangles.

\begin{definition}
  A \defemph{decorated \(\slg\)-coloring} assigns each segment \(i\) a group element \(g_i \in \slg\) and a \(1\)-dimensional subspace \(L_i \subset \mathbb{C}^{2}\) invariant under \(g_i\), subject to the rule that
  \begin{gather}
    \label{eq:decorated-positive}
    L_{2'} = L_2 g_1  \text{ (positive)}
    \\
    \label{eq:decorated-negative}
    L_{1'} = L_{1} g_{2} \text{ (negative)}
  \end{gather}
  These are summarized in \cref{fig:decorated-quandle-rule}.
\end{definition}

In our example we cannot multiply two segment colors because \((g, L)\) and \((h, L')\) do not determine an eigenspace of \(gh\), but we can conjugate them.
As such the set of pairs  \((g, L)\) with the binary operation
\[
  (g, L) \triangleleft (h, L') = (h^{-1} g h, L h)
\]
form an algebraic structure \(\mathcal{Q}\) called a \defemph{quandle}.
Quandles generalize the conjugation structure of a group and appear in many places in knot theory.
Decorated \(\slg\)-colorings can be characterized as determining homomorphism from the fundamental quandle of the tangle to \(\mathcal{Q}\).

\begin{proposition}
  Decorated representations \(\rho : \pi(D) \to \slg\) are in bijection with decorated \(\slg\)-colorings of \(D\).
\end{proposition}

Because \(\slg\)  is the double cover of the isometry group \(\pslg\) of hyperbolic space a choice of \(\rho : \pi(T) \to \slg\) determines a (generalized) hyperbolic structure on the tangle complement.
However, this description is rather abstract, and frequently one wants to instead describe the hyperbolic structure by gluing together geometric ideal tetrahedra.
Any tangle diagram determines an \defemph{octahedral decomposition} \cite{Kim2016,McPhailSnyder2022} of its complement in a standard way, which gives the topological data of the gluing.
The geometric information is determined by complex numbers called shape parameters, which are in turn determined by an octahedral coloring of the digram as defined below.
We discuss the shape parameters further in \cref{sec:potential}.

\begin{definition}
  \label{def:octahedral-stuff}
  An \defemph{octahedral color} is a triple \(\chi = (a, b, m) \in (\mathbb{C} \setminus \set{0} )^3\) of nonzero complex numbers.
  When we work with an indexed set \(\set{\chi_{i}}_{i \in I}\) of colors we typically write \(a_i, b_i, m_i\) for their components.
  An \defemph{octahedral coloring} of a tangle diagram \(D\) is an assignment \(i \mapsto \chi_i = (a_i, b_i, m_i)\) of an octahedral color to every segment \(i\) so that at every positive crossing
  \begin{gather}
    \label{eq:a-transf-positive}
    \begin{aligned}
      a_{1'}
      &=
      a_1 A^{-1}
      \\
      a_{2'}
      &=
      a_2 A \text{ where}
      \\
      A &= 1 - \frac{m_1 b_1}{b_2} \left(1 - \frac{a_1}{m_1}\right)\left(1 - \frac{1}{m_2 a_2}\right)
    \end{aligned}
    \\
    \label{eq:b-transf-positive}
    \begin{aligned}
      b_{1'}
      &=
      \frac{m_2 b_2}{m_1}
      \left(
        1 - m_2 a_2 \left( 1 - \frac{b_2}{m_1 b_1} \right)
      \right)^{-1}
      \\
      b_{2'}
      &=
      b_1
      \left(
        1 - \frac{m_1}{a_1}\left( 1 - \frac{b_2}{m_1 b_1} \right)
      \right)
    \end{aligned}
    \\
    \label{eq:m-transf-positive}
    \begin{aligned}
      m_{1'}
      &= m_1
      &
      m_{2'}
      &= m_2
    \end{aligned}
  \end{gather}
  and at every negative crossing
  \begin{gather}
    \label{eq:a-transf-negative}
    \begin{aligned}
      a_{1'}
      &=
      a_1 \tilde A^{-1} 
      \\
      a_{2'}
      &=
      a_2 \tilde A \text{ where}
      \\
      \tilde A 
      &=
      1 - \frac{b_2}{m_1 b_1}\left(1 - m_1 a_1 \right)\left(1 - \frac{m_2}{a_2}\right).
    \end{aligned}
    \\
    \label{eq:b-transf-negative}
    \begin{aligned}
      b_{1'}
      &=
      \frac{m_2 b_2}{m_1}
      \left(
        1 - \frac{a_2}{m_2} \left( 1 - \frac{m_1 b_1}{b_2} \right)
      \right)
      \\
      b_{2'}
      &=
      b_1
      \left(
        1 - \frac{1}{m_1 a_1} \left( 1 - \frac{m_1 b_1}{b_2} \right)
      \right)^{-1}
    \end{aligned}
    \\
    \label{eq:m-transf-negative}
    \begin{aligned}
      m_{1'}
      &= m_1
      &
      m_{2'}
      &= m_2
    \end{aligned}
  \end{gather}
\end{definition}

Our goal is to obtain an octahedral coloring from a decorated \(\slg\)-coloring.
The relevant geometry of the ideal tetrahedra is most naturally described in terms of the face-pairing maps of the triangulation.
These define a representation of a fundamental \emph{groupoid} with one basepoint for each tetrahedron.
This groupoid is closely related (but not quite the same as) the fundamental groupoid studied in \cref{sec:holonomy}.
In contrast, the Wirtinger generators are global and have a single common basepoint.
To relate these to the fundamental groupoid we need to keep track of change-of-basepoint data using additional colors called shadows.

\begin{definition}
  Let \(i \mapsto g_i\) be a \(\slg\)-coloring of a tangle diagram \(D\) (possibly decorated).  
  A \defemph{shadow coloring} of \((D,g)\) is an assignment \(j \mapsto u_j\) of nonzero column vectors in \(\mathbb{C}^{2}\) to each region, subject to the rule
  \[
    u_{j'} = g_i u_j
  \]
  when \(j'\) is below \(j\) across \(i\) as in \cref{fig:region-orientation-rule}.
\end{definition}

\begin{lemma}
  \label{thm:shadow-colors-free}
  For a fixed \(\slg\)-coloring the space of shadow colorings is nonempty and in bijection with \(\mathbb{C}^2 \setminus \set{0}\).
\end{lemma}

\begin{marginfigure}
  \centering
\begingroup%
  \makeatletter%
  \providecommand\color[2][]{%
    \errmessage{(Inkscape) Color is used for the text in Inkscape, but the package 'color.sty' is not loaded}%
    \renewcommand\color[2][]{}%
  }%
  \providecommand\transparent[1]{%
    \errmessage{(Inkscape) Transparency is used (non-zero) for the text in Inkscape, but the package 'transparent.sty' is not loaded}%
    \renewcommand\transparent[1]{}%
  }%
  \providecommand\rotatebox[2]{#2}%
  \newcommand*\fsize{\dimexpr\f@size pt\relax}%
  \newcommand*\lineheight[1]{\fontsize{\fsize}{#1\fsize}\selectfont}%
  \ifx\svgwidth\undefined%
    \setlength{\unitlength}{108.86558533bp}%
    \ifx\svgscale\undefined%
      \relax%
    \else%
      \setlength{\unitlength}{\unitlength * \real{\svgscale}}%
    \fi%
  \else%
    \setlength{\unitlength}{\svgwidth}%
  \fi%
  \global\let\svgwidth\undefined%
  \global\let\svgscale\undefined%
  \makeatother%
  \begin{picture}(1,0.76955296)%
    \lineheight{1}%
    \setlength\tabcolsep{0pt}%
    \put(0,0){\includegraphics[width=\unitlength,page=1]{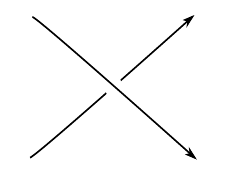}}%
    \put(0.0215491,0.73178434){\makebox(0,0)[lt]{\lineheight{1.25}\smash{\begin{tabular}[t]{l}$g_1$\end{tabular}}}}%
    \put(0.02154686,-0.02603094){\makebox(0,0)[lt]{\lineheight{1.25}\smash{\begin{tabular}[t]{l}$g_2$\end{tabular}}}}%
    \put(0.75869458,0.73526858){\makebox(0,0)[lt]{\lineheight{1.25}\smash{\begin{tabular}[t]{l}$g_1^{-1} g_2 g_1$\end{tabular}}}}%
    \put(0.75869458,-0.0260219){\makebox(0,0)[lt]{\lineheight{1.25}\smash{\begin{tabular}[t]{l}$g_1$\end{tabular}}}}%
    \put(0.44604092,0.63225879){\makebox(0,0)[lt]{\lineheight{1.25}\smash{\begin{tabular}[t]{l}$u$\end{tabular}}}}%
    \put(0.10429751,0.41869915){\makebox(0,0)[lt]{\lineheight{1.25}\smash{\begin{tabular}[t]{l}$g_1 u$\end{tabular}}}}%
    \put(0.36515411,0.13240361){\makebox(0,0)[lt]{\lineheight{1.25}\smash{\begin{tabular}[t]{l}$g_2 g_1 u$\end{tabular}}}}%
    \put(0.73009579,0.4177686){\makebox(0,0)[lt]{\lineheight{1.25}\smash{\begin{tabular}[t]{l}$g_1^{-1} g_2 g_1 u$\end{tabular}}}}%
  \end{picture}%
\endgroup%

  \caption{We can consistently assign colors to the regions around a positive crossing because \(g_2 g_1 = g_1 (g_1^{-1} g_2 g_1)\). Negative crossings are similar.}
  \label{fig:shadow-rule-around}
\end{marginfigure}

\begin{proof}
  To see that shadow colorings exist it is enough to check that the rules are consistent around each crossing.
  For a positive crossing (shown in \cref{fig:shadow-rule-around}) this is because \(g_2 g_1 = g_1 (g_1^{-1} g_2 g_1)\), and negative crossings are similar.

  Once the color \(u_j\) of a single region \(j\) is chosen the rules determine a unique color for every other region, so the space of shadow colorings is parametrized by \(u_j \in \mathbb{C}^{2} \setminus \set{0}\).
\end{proof}

\begin{definition}
  \label{def:associated-octahedral}
  Consider a \(\slg\)-shadow-colored digram \((D, g, u)\).
  For each segment \(i\), let \(v_i\) be a nonzero element of \(L_i\); our definition is independent of this choice.
  Let \(m_i\) be the \emph{inverse}%
  \note{
    This choice is made to match the conventions of \cite{McPhailSnyderAlgebra}, which were established before this geometric interpretation was fully understood.
  }
  eigenvalue of \(g_i\) on \(v_i\):
   \[
     v_i g_i = m_{i}^{-1} v_{i}.
  \]
  We say \((D, g, u)\) is  \defemph{admissible} if 
  \begin{gather*}
    \det( u_{j}, \evec{2} ) \ne 0 \text{ for each region \(j\), and}
    \\
    v_i \evec{2}, v_i u_{\upr{i}} \ne 0 \text{ for each segment \(i\).}
  \end{gather*}
  In this case we can assign segment \(i\) below region \(\upr{i}\) and above region \(\dnr{i}\) the octahedral color
  \begin{equation}
    \label{eq:chi-formula}
    \chi_i =
    \left(
      \frac{
        \det( u_{\dnr{i}} , \evec{2} )
        }{
        \det( u_{\upr{i}} , \evec{2} )
      }
      ,
      -
      \frac{
        v_i \evec{2}
      }{
        v_i u_{\upr{i}}
      }
      ,
      m_i
    \right)
  \end{equation}
  and we call this the octahedral coloring \defemph{associated} to \((D, g, u)\).
\end{definition}

\begin{theorem}
  The coloring \(i \mapsto \chi_i\) defined by \eqref{eq:chi-formula} is an octahedral coloring: it satisfies the conditions of \cref{def:octahedral-stuff}.
\end{theorem}

\begin{proof}
  We can parametrize the space of all \(\slg\) decorated shadow colorings of a crossing by the coordinates of \(g_1, g_2, L_1, L_2, u\) (in the notation of \cref{fig:decorated-quandle-rule}), say by using homogeneous coordinates for \(L_1, L_2 \in \psp\).
  The components \(a_i, b_i, m_i\) of the associated octahedral coloring are rational functions in these.
  One can check that these rational functions agree with \cref{eq:a-transf-positive,eq:b-transf-positive,eq:m-transf-positive}, and similarly \cref{eq:a-transf-negative,eq:b-transf-negative,eq:b-transf-negative} for a negative crossings.
\end{proof}

Our proof is algebraic and not very enlightening.
What we really want to know is that the holonomy of the octahedral coloring coming from the decomposition into geometric ideal tetrahedra agrees with the representation \(\rho\), which is \cref{thm:holonomies-match} of \cref{sec:holonomy}.

\section{The holonomy of an octahedral coloring}
\label{sec:holonomy}

In this section we define the holonomy representation of an octahedral coloring and show that it is compatible with our method of assigning an octahedral coloring to a shadowed decorated \(\slg\)-coloring, thus justifying \cref{def:associated-octahedral}.
To define the holonomy we must first introduce a groupoid presentation of \(\pi(D)\).

Recall the characterization of a group \(G\) as a category with a single object \(\bullet\) and an invertible morphism \(g : \bullet \to \bullet\) for each \(g \in G\).
A \defemph{groupoid} is the natural generalization: there can be more than one object, but all morphisms are still invertible.
For example, instead of the fundamental group of a topological space with a single basepoint, one can consider a fundamental groupoid with multiple basepoints whose morphisms \(p \to q\) are homotopy classes of paths from \(p\) to \(q\).
Two paths (morphisms) \(f\) and \(g\) are composable only when the endpoint (codomain) of \(f\) is the startpoint (domain) of \(g\).
For a tangle complement it is natural to put one basepoint in each region.

\begin{marginfigure}
\begingroup%
  \makeatletter%
  \providecommand\color[2][]{%
    \errmessage{(Inkscape) Color is used for the text in Inkscape, but the package 'color.sty' is not loaded}%
    \renewcommand\color[2][]{}%
  }%
  \providecommand\transparent[1]{%
    \errmessage{(Inkscape) Transparency is used (non-zero) for the text in Inkscape, but the package 'transparent.sty' is not loaded}%
    \renewcommand\transparent[1]{}%
  }%
  \providecommand\rotatebox[2]{#2}%
  \newcommand*\fsize{\dimexpr\f@size pt\relax}%
  \newcommand*\lineheight[1]{\fontsize{\fsize}{#1\fsize}\selectfont}%
  \ifx\svgwidth\undefined%
    \setlength{\unitlength}{140.23033333bp}%
    \ifx\svgscale\undefined%
      \relax%
    \else%
      \setlength{\unitlength}{\unitlength * \real{\svgscale}}%
    \fi%
  \else%
    \setlength{\unitlength}{\svgwidth}%
  \fi%
  \global\let\svgwidth\undefined%
  \global\let\svgscale\undefined%
  \makeatother%
  \begin{picture}(1,0.37438403)%
    \lineheight{1}%
    \setlength\tabcolsep{0pt}%
    \put(0,0){\includegraphics[width=\unitlength,page=1]{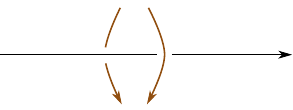}}%
    \put(0.57762114,0.26741715){\makebox(0,0)[lt]{\lineheight{1.25}\smash{\begin{tabular}[t]{l}$x_i^+$\end{tabular}}}}%
    \put(0.25672045,0.26741715){\makebox(0,0)[lt]{\lineheight{1.25}\smash{\begin{tabular}[t]{l}$x_i^-$\end{tabular}}}}%
    \put(0,0){\includegraphics[width=\unitlength,page=2]{tangle-groupoid-generators.pdf}}%
    \put(0.64180094,0.10696684){\makebox(0,0)[lt]{\lineheight{1.25}\smash{\begin{tabular}[t]{l}$i$\end{tabular}}}}%
  \end{picture}%
\endgroup%

  \caption{Generators of the fundamental groupoid \(\gpoid{D}\) of a tangle diagram.}
  \label{fig:tangle-groupoid-generators}
\end{marginfigure}

\begin{definition}
  \label{def:tangle-groupoid}
  Let \(D\) be a tangle diagram.
  The \defemph{fundamental groupoid} \(\gpoid{D}\) of \(D\) has
  \begin{itemize}
    \item one object (i.e.\@ basepoint) for each  region of \(D\),
    \item two generators \(x_i^{\pm}\) for each segment \(i\), representing paths above and below the segment (see \cref{fig:tangle-groupoid-generators}), and
    \item three relations for each crossing:

  \begin{marginfigure}
\begingroup%
  \makeatletter%
  \providecommand\color[2][]{%
    \errmessage{(Inkscape) Color is used for the text in Inkscape, but the package 'color.sty' is not loaded}%
    \renewcommand\color[2][]{}%
  }%
  \providecommand\transparent[1]{%
    \errmessage{(Inkscape) Transparency is used (non-zero) for the text in Inkscape, but the package 'transparent.sty' is not loaded}%
    \renewcommand\transparent[1]{}%
  }%
  \providecommand\rotatebox[2]{#2}%
  \newcommand*\fsize{\dimexpr\f@size pt\relax}%
  \newcommand*\lineheight[1]{\fontsize{\fsize}{#1\fsize}\selectfont}%
  \ifx\svgwidth\undefined%
    \setlength{\unitlength}{144bp}%
    \ifx\svgscale\undefined%
      \relax%
    \else%
      \setlength{\unitlength}{\unitlength * \real{\svgscale}}%
    \fi%
  \else%
    \setlength{\unitlength}{\svgwidth}%
  \fi%
  \global\let\svgwidth\undefined%
  \global\let\svgscale\undefined%
  \makeatother%
  \begin{picture}(1,0.67307788)%
    \lineheight{1}%
    \setlength\tabcolsep{0pt}%
    \put(0,0){\includegraphics[width=\unitlength,page=1]{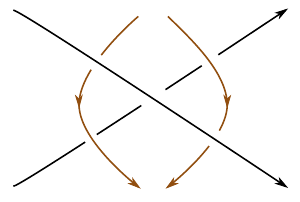}}%
    \put(-0.0032803,0.63811369){\makebox(0,0)[lt]{\lineheight{1.25}\smash{\begin{tabular}[t]{l}$1$\end{tabular}}}}%
    \put(-0.0032803,0.03553828){\makebox(0,0)[lt]{\lineheight{1.25}\smash{\begin{tabular}[t]{l}$2$\end{tabular}}}}%
    \put(0.97038872,0.04534311){\makebox(0,0)[lt]{\lineheight{1.25}\smash{\begin{tabular}[t]{l}$1'$\end{tabular}}}}%
    \put(0.97038872,0.64305345){\makebox(0,0)[lt]{\lineheight{1.25}\smash{\begin{tabular}[t]{l}$2'$\end{tabular}}}}%
    \put(0.07148831,0.32196937){\makebox(0,0)[lt]{\lineheight{1.25}\smash{\begin{tabular}[t]{l}$x_1^- x_2^+$\end{tabular}}}}%
    \put(0.79194551,0.32196937){\makebox(0,0)[lt]{\lineheight{1.25}\smash{\begin{tabular}[t]{l}$x_{2'}^+ x_{1'}^-$\end{tabular}}}}%
  \end{picture}%
\endgroup%

    \caption{Deriving the middle relation at a crossing.}
    \label{fig:tangle-groupoid-relations}
  \end{marginfigure}

  \begin{gather}
    \label{eq:groupoid-relation-above}
    x_1^+ x_2^+ = x_{2'}^+ x_{1'}^+
    \\
    \label{eq:groupoid-relation-below}
    x_1^- x_2^- = x_{2'}^- x_{1'}^-
    \\
    \label{eq:groupoid-relation-thru}
    \begin{cases}
      x_1^{-} x_{2}^{+}
      =
      x_{2'}^{+} x_{1'}^{-}
      &
      \text{ for a positive crossing, or}
      \\
      x_1^{+} x_{2}^{-}
      =
      x_{2'}^{-} x_{1'}^{+}
      &
      \text{ for a negative crossing.} 
    \end{cases}
  \end{gather}
  where the indices \(1,2, 1', 2'\) refer to the segments at the crossing as in \cref{fig:crossing-types}.
  \end{itemize}

\end{definition}

One can derive relations (\ref{eq:groupoid-relation-above},\ref{eq:groupoid-relation-below},\ref{eq:groupoid-relation-thru}) by drawing diagrams like \cref{fig:tangle-groupoid-relations}.
It is possible to prove that \(\Pi(D)\) is equivalent to the fundamental group \(\pi(T)\) and thus to \(\pi(D)\) by using the groupoid version \cite[Chapter 9]{Brown2006} of the van Kampen theorem.
We do not need this result so we omit the details.

For an octahedral color \(\chi = (a, b, m)\) we write
\begin{gather}
  \label{eq:chi-up-hol}
  \upf{\chi} \defeq
  \begin{bmatrix}
    a & 0 \\
    (a - 1/m)/b & 1
  \end{bmatrix}
  \\
  \label{eq:chi-down-hol}
  \dnf{\chi} \defeq
  \begin{bmatrix}
    1 & (a - m)b \\
    0 & a
  \end{bmatrix}
  \intertext{and}
  \label{eq:chi-around-hol}
  \aroundhol{\chi}
  \defeq
  \upf{\chi} \dnf{\chi}^{-1}
  =
  \begin{bmatrix}
    a & -(a-m) b \\
    (a-1/m)/b & m + m^{-1} -a
  \end{bmatrix}
\end{gather}

\begin{definition}
  \label{def:holonomy}
  Let \(D\) be a tangle diagram with an octahedral coloring \(\chi\).
  The \defemph{holonomy representation} is the representation (i.e.\@ functor)
  \[
    \Hol{\chi} : \gpoid{D} \to \glg
  \]
  defined by%
  \note{
    One could instead obtain a representation \(\gpoid{D} \to \pslg\) by dividing by an arbitrary square root of \(a\).
    A lift \(\pi(D) \to \slg\) can be recovered by choosing a square root.
    This approach is geometrically more natural but algebraically inconvenient.
  }
  \begin{equation}
    \Hol{\chi}(x^{+}_{i}) =
    \upf{\chi_i}
    \text{ and }
    \Hol{\chi}(x^{-}_{i}) =
    \dnf{\chi_i}
  \end{equation}
  for each segment \(i\) of \(D\).
\end{definition}

\begin{lemma}
  \(\Hol{\chi}\) is well-defined.
\end{lemma}

\begin{proof}
  It is elementary (but tedious) to check that whenever the octahedral colors at a crossing are related as in \cref{def:octahedral-stuff} the holonomy matrices satisfy the relations of \cref{def:tangle-groupoid}.
\end{proof}

The holonomy representation is related to the one generated by the face maps of the octahedral decomposition, as discussed in \cite[Section 3]{McPhailSnyder2022}.
To relate it to representations of \(\pi(D)\) we define a functor \(\pi(D) \to \Pi(D)\) constructed using a natural family of paths in \(\Pi(D)\):
\begin{definition}
  \label{def:shadows}
  Let \(D\) be a tangle diagram with topmost region \(0\) and let \(j\) be a region of \(D\).
  An \defemph{over path} \(s_j^+\) is a path \(0 \to j\) that passes over each strand.
  It is clear it unique up to homotopy.
\end{definition}

To write these paths as products of generators, choose a path of adjacent regions from \(0\) to \(j\).
If that path crosses segments \(i_1, \dots, i_k\) in that order, then
\begin{equation}
  \label{eq:shadow-path-rule}
  s_j^+ = \left(x_{i_1}^+\right)^{\epsilon_1} \cdots \left(x_{i_k}^+\right)^{\epsilon_k}
\end{equation}
where \(\epsilon_i = + 1\) if segment \(i\) is oriented left-to-right as we cross it from the top and \(-1\) otherwise, as in \cref{fig:region-orientation-rule}.
Independence of \cref{eq:shadow-path-rule} from the choice of path follows from \cref{eq:groupoid-relation-above,eq:groupoid-relation-below}.

\begin{theorem}
  \label{thm:functor}
  Let \(D\) be a tangle diagram.
  For each segment \(i\) write \(\upr{i}\) for the region above \(i\) and set
  \[
    \mathcal{F}(w_i) \defeq s^+_{\upr{i}} \left[x_i^{+} x_i^{-} \right] (s_{\upr{i}}^+)^{-1} 
  \]
  where \(w_i\) is the Wirtinger generator associated to \(i\).
  Then \(\mathcal{F}\) is a well-defined homomorphism (i.e.,\@ functor) \(\pi(D) \to \gpoid{D}\).
\end{theorem}

An example is given in in \cref{fig:path-factorization}.

\begin{proof}
   We need to check that \(\mathcal{F}\) respects the relations of \(\pi(D)\) coming from each crossing.
   Consider a positive crossing whose topmost region has over path \(s^+\) and whose segments are labeled as in \cref{fig:crossing-types}.
   We have
   \[
     \mathcal{F}(w_{2'})
     =
     s^+
     x_{2'}^{+}
     \left(x_{2'}^{-}\right)^{-1}
     \left(s^+\right)^{-1}
   \]
   while
   \begin{align*}
     &
     \mathcal{F}(w_{1}^{-1} w_{2} w_{1})
     \\
     &=
     \left[
     s^+
       x_1^- \left(x_1^+\right)^{-1}
     \left(s^+\right)^{-1}
     \right]
     s^+
     x_{1}^{+}
     x_2^+ \left(x_2^-\right)^{-1}
     \left(s^{+} x_1^{+}\right)^{-1}
     \left[
     s^+
       x_1^+ \left(x_1^-\right)^{-1}
     \left(s^+\right)^{-1}
     \right]
     \\
     &=
     s^+
     x_1^-
     x_2^+ \left(x_2^-\right)^{-1}
     \left(x_1^-\right)^{-1}
     \left(s^+\right)^{-1}
   \end{align*}
   and writing \eqref{eq:groupoid-relation-thru} and \eqref{eq:groupoid-relation-below} as
   \[
     x_{2'}^{+}
     =
     x_1^{-} x_2^{+} \left( x_{1'}^{-}\right)^{-1}
     \text{ and }
     \left(x_{2'}^{-}\right)^{-1}
     =
     x_{1'}^{-}
     \left( x_2^{-} \right)^{-1}
     \left( x_1^{-} \right)^{-1}
   \]
   gives the relation \(\mathcal{F}(w_{2'}) = \mathcal{F}(w_1^{-1} w_2 w_1)\).
   The relation at a negative crossing follows from a similar computation.
\end{proof}

\begin{figure}
  \centering
  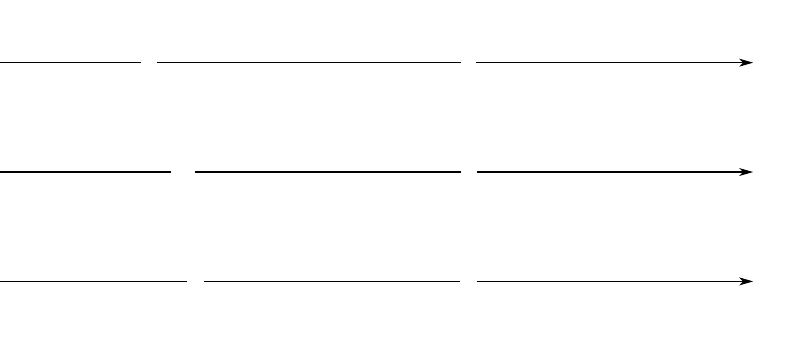
\caption{
  This tangle has regions \(0, 1, 2, 3\) and segments \(1, 2, 3\).
  The over path for region \(2 = \upr{3}\) is \(s^+ = x_1^+ x_2^+\), so the Wirtinger generator \(w_3 \in \pi(D)\) of segment \(3\) is mapped to \( \mathcal{F}(w_3) = s^+ x_3^+ (x_3^-)^{-1} (s^+)^{-1} = x_1^+ x_2^+ x_3^+ \left(x_3^- \right)^{-1} \left(x_2^+\right)^{-1} \left(x_1^+\right)^{-1} \in  \gpoid{D}.\)
}
  \label{fig:path-factorization}
\end{figure}

According to the theorem, an octahedral coloring \(i \mapsto \chi_i\) has the same holonomy as an  \(\slg\) representation \(\rho\) if for each segment \(i\)
\[
  \rho(w_i) = \Hol{\chi}(s_{\upr{i}}^+) \aroundhol{\chi_i} \Hol{\chi}(s_{\upr{i}}^+)^{-1}.
\]
Here \(\rho(w_i) = g_i\) is the image of the corresponding Wirtinger generator, i.e.\@ the color strand \(i\) in the associated \(\slg\)-coloring.
This gives an inductive method for finding \(\chi\).
For example, consider the tangle in \cref{fig:path-factorization} with segments \(1, 2, 3\) and regions \(0, 1, 2, 3\).
The octahedral color \(\chi_1\) of strand \(1\) is determined by
\[
  g_1 = \Hol{\chi}(s_0^+) \aroundhol{\chi_1} \Hol{\chi}(s_0^+)^{-1}
  = \aroundhol{\chi_1}
\]
as the over path \(s_0^+\) for the topmost region is trivial.
We can find such a \(\chi_1\) whenever \(g_1\) lies in the subset 
\[
  Y
  =
  \set{ 
    \begin{bmatrix}
      a & -(a-m) b \\
      (a-1/m)/b & m + m^{-1} -a
    \end{bmatrix}
    \given
    a, b, m \in \mathbb{C} \setminus \{0\}
  }
  \subset \slg
\]
of matrices of the form \eqref{eq:chi-around-hol}.
Moving down the diagram, \(\chi_2\) satisfies
\[
  g_2 = \Hol{\chi}(s_1^+) \aroundhol{\chi_2} \Hol{\chi}(s_1^+)^{-1}
  = \upf{\chi_1} \aroundhol{\chi_2} \upf{\chi_1}^{-1}
\]
and again this is possible when
\[
  \upf{\chi_1}^{-1} g_2 \upf{\chi_1} \in Y.
\]

More generally this method can be used to inductively find an octahedral coloring agreeing with \(\rho\);
this leads to the \defemph{generic biquandle factorizations} of of \textcite{Blanchet2018}.%
\note{
  In particular, octahedral colorings are closely related to the factorization of \(\slg\)  in \cite[Example 5.2]{Blanchet2018}.
  The connection to quantum groups in \cite[Section 6]{Blanchet2018} is discussed further in \cite{McPhailSnyder2022,McPhailSnyderAlgebra}.
}
However, it is not immediately clear for which \(\rho\) this is possible, because the condition that
\[
  \Hol{\chi}(s_{\upr{i}}^+)^{-1} g_i \Hol{\chi}(s_{\upr{i}}^+)
  \in Y
\]
requires already determining the earlier colors in the over path
\[
  \Hol{\chi}(s_{\upr{i}}^+) 
  =
  \upf{\chi_{i_1}} \cdots \upf{\chi_{i_k}}.
\]

In addition, this definition is asymmetric: the segments adjacent to the top region are treated differently because its over path is always trivial.
We can change this by modifying our notion of what it means for \(\Hol{\chi}\) to agree with \(\rho\).
Extend our earlier notation by writing
\[
  \upf{
    \begin{bmatrix}
      u^1 \\ u^2
    \end{bmatrix}
  }
  =
  \begin{bmatrix}
    u^{1} & 0 \\
    u^{2} & 1
  \end{bmatrix}
\]
and define
\[
  \hol{\chi, u_0}(w_i) = \upf{u_0} \Hol{\chi}(\mathcal{F}(w_i)) \upf{u_0}^{-1}.
\]
This makes sense as long as the first entry of \(u_0\) is nonzero, which it is in any admissible shadow coloring.
When \(u_0 = \evec{1}\) we simply have \(\hol{\chi, \evec{1}} = \Hol{\chi} \mathcal{F} \), while the other representations \(\hol{\chi, u_0}\) are conjugate by the matrices \(\upf{u_0}\).
We now rephrase our problem as asking whether \(\rho\) agrees with some \(\hol{\chi, u_0}\).
This does not fundamentally change the problem because we are mostly interested in whether \(\rho\) is \emph{conjugate} to the holonomy of an octahedral coloring.

\begin{bigtheorem}
  \label{thm:holonomies-match}
  Let \(D\) be a tangle diagram and \(i \mapsto g_i, j \mapsto u_j\) be a decorated shadow \(\slg\)-coloring of \(D\) with \(u_0\) the shadow of the topmost region.
  Let \(\chi\) be the associated octahedral coloring of \cref{def:associated-octahedral}.
  Then its holonomy agrees with \(\rho\) in the sense that
  \[
    \hol{\chi, u_0}(w_i) = \rho(w_i)
  \]
  for each Wirtinger generator \(w_i\).
\end{bigtheorem}

We begin with some matrix computations that reduce the proof to showing that the shadow coloring keeps track of the over path holonomies.
\begin{lemma}
  \label{thm:shadow-lemma}
  Suppose 
  \[
    \upf{u} \aroundhol{\chi} \upf{u}^{-1} = g.
  \]
  for an octahedral color \(\chi\), a column vector \(u \in \mathbb{C}^{2}\), and a matrix \(g \in \slg\).
  Then
  \begin{equation*}
    \chi =
    \left(
      \frac{
        \det( g u , \evec{2} )
      }{
        \det( u , \evec{2} )
      }
      ,
      -
      \frac{
        v \evec{2}
      }{
        v u
      }
      ,
      m
    \right)
  \end{equation*}
  where \(v\) is a nonzero \emph{row} vector with \(vg = m^{-1} v\).
\end{lemma}

\begin{proof}
  Write \(\chi = (a,b, m)\).
  By \cref{eq:chi-around-hol} we have
  \[
    \begin{bmatrix}
      -1 & b
    \end{bmatrix}
    \aroundhol{\chi}
    =
    m^{-1}
    \begin{bmatrix}
      -1 & b
    \end{bmatrix}
  \]
  so \(
    \begin{bmatrix}
      -1 & b
    \end{bmatrix}
  \)
  spans the \(m^{-1}\)-eigenspace of \(\aroundhol{\chi}\).
  (When \(m = \pm 1\) it spans the single eigenspace of \(\tau{\chi}\).)
  Since \(v\) spans the \(m^{-1}\) eigenspace of \(g = \upf{u}^{-1} \aroundhol{\chi} \upf{u}\) we conclude that
  \[
    \alpha
    \begin{bmatrix}
      -1 & b
    \end{bmatrix}
    =
    v
    \upf{u}
    =
    \begin{bmatrix}
      v u & v \evec{2}
    \end{bmatrix}
  \]
  for some nonzero \(\alpha \in \mathbb{C}\), which implies
  \[
    b = -
    \frac{
      v \evec{2}
      }{
      v u
    }
  \]
  as claimed.
  It remains to establish the formula for \(a\).
  Set
  \[
    u = \begin{bmatrix}
      u^1 \\ u^2
    \end{bmatrix}.
  \]
  As for any column vector \(\det(u, \evec{2}) = u^1 = \evec{1}^{\transp} u\).
  In addition,
  \(
    \upf{u}^{-1} u =  \evec{1}
  \)
  because
  \(
    \upf{u} = \begin{bmatrix}
      u & \evec{2}
    \end{bmatrix}
  \)
  and \(\dnf{\chi} \evec{1} = \evec{1}\) because \(\dnf{\chi}\) is lower-triangular.
  We now compute
  \begin{align*}
    \det(u, \evec{2}) 
    &=
    \evec{1}^{\transp}
    g u
    \\
    &=
    \evec{1}^{\transp}
    \upf{u} \upf{\chi} \dnf{\chi} \upf{u}^{-1} u 
    \\
    &=
    \evec{1}^{\transp}
    \upf{u} \upf{\chi} \dnf{\chi} \evec{1}
    \\
    &=
    \evec{1}^{\transp}
    \upf{u} \upf{\chi} \evec{1}
    \\
    &=
    \begin{bmatrix}
      u^1 & 0 
    \end{bmatrix}
    \begin{bmatrix}
      a \\
      (a-1/m)/b
    \end{bmatrix}
    \\
    &= u^1 a.
  \end{align*}
\end{proof}

\begin{proof}[Proof of \cref{thm:holonomies-match}]
  Our goal is to show that for each segment \(i\),
  \[
    \hol{\chi, u_0}(w_i)
    =
    \upf{u_0} \Hol{\chi}(s_{\upr{i}}^+)
    \aroundhol{\chi_i}
    \left( \upf{u_0} \Hol{\chi}(s_{\upr{i}}^+) \right)^{-1}
    =
    g_i.
  \]
  By \cref{thm:shadow-lemma} it is sufficient to check that
  \begin{equation}
    \label{eq:shadows-match-over}
    \upf{u_j} = \upf{u_0} \Hol{\chi}(s_j^+) 
  \end{equation}
  for each region \(j\).

  We prove this inductively.
  For the topmost region \(0\), \(s_0^+ = 1\) is trivial and \cref{eq:shadows-match-over} is tautological.
  Below we show that if \cref{eq:shadows-match-over} holds for one region then it holds for all adjacent regions, which establishes the theorem.

  Suppose
  \[
    \upf{u_j} = \upf{u_0} \Hol{\chi}(s_j^+) 
  \]
  for a region \(j\) and let  \(j'\) be adjacent across segment \(i\).
  For notational simplicity we assume \(j'\) is below \(j\) (as in \cref{fig:region-orientation-rule}) but a similar argument works for the other orientation.
  By definition,
  \(
    u_{j'} = g_i u_{j}
  \)
  so by the inductive hypothesis
  \[
    u_{j'}
    =
    \upf{u_j} \aroundhol{\chi_i} \upf{u_j}^{-1} u_{j}.
  \]
  Then as in the proof of \cref{thm:shadow-lemma}
  \[
    u_{j'}
    =
    \upf{u_j} \aroundhol{\chi_i} \upf{u_j}^{-1} u_{j}
    =
    \upf{u_j} \upf{\chi_i} \dnf{\chi_i} \evec{1}
    =
    \upf{u_j} \upf{\chi_i} \evec{1}.
  \]
  Now by the inductive hypothesis and the definition of over path
  \[
    \upf{u_j} \upf{\chi_i}
    =
    \upf{u_0} \Hol{\chi}(s_j^+) \upf{\chi_i}
    =
    \upf{u_0} \Hol{\chi}(s_{j'}^+) 
  \]
  so
  \[
    u_{j'}
    =
    \upf{u_0} \Hol{\chi}(s_{j'}^+) \evec{1}
  \]
  and because \(\upf{u} \evec{1} = u\) we conclude \(\upf{u_{j'}} = \upf{u_0} \Hol{\chi}(s_{j'}^+)\) as claimed.
\end{proof}

\section{Gauge transformations}
\label{sec:gauge-transformations}

In this section we discuss how gauge transformation (conjugation) of the representation \(\rho\) interacts with the decorations and shadows.
In this context it is natural to break gauge transformations into two separate types.
We then use gauge transformations to prove that every decorated representation is gauge-equivalent to an admissible one (i.e.\@ one that is the holonomy of an octahedral coloring).

\begin{definition}
  Let \(i \mapsto (g_i, L_i)\), \(j \mapsto u_j\) be a decorated shadow \(\slg\) coloring of a tangle diagram \(D\).
  A \defemph{gauge transformation} of it is a coloring of the form
  \begin{itemize}
    \item[(A)] \(i \mapsto (h^{-1} g_i h, L_i h)\) and \(j \mapsto g^{-1} u_j\)
    \item[(B)] \(i \mapsto (g_i, L_i)\) and \(j \mapsto h^{-1} u_j\)
  \end{itemize}
  for some \(h \in \slg\).
  We say the new colorings are the images of \defemph{type (A)} and  \defemph{type (B)} gauge transformations  \defemph{by \(h\)}.
\end{definition}
It is not hard to see that the gauge transformation of a coloring is still a coloring and that the corresponding representation \(\rho : \pi(D) \to \slg\) is conjugated under a type (A) transformation and unchanged under type (B).
We can use type (B) transformations to choose the shadow of a single region arbitrarily as in \cref{thm:shadow-colors-free}.

\begin{figure}
  \subcaptionbox{ \(D'\) is obtained from \(D\) by a type (A) gauge transformation by \(h\).}{
\begingroup%
  \makeatletter%
  \providecommand\color[2][]{%
    \errmessage{(Inkscape) Color is used for the text in Inkscape, but the package 'color.sty' is not loaded}%
    \renewcommand\color[2][]{}%
  }%
  \providecommand\transparent[1]{%
    \errmessage{(Inkscape) Transparency is used (non-zero) for the text in Inkscape, but the package 'transparent.sty' is not loaded}%
    \renewcommand\transparent[1]{}%
  }%
  \providecommand\rotatebox[2]{#2}%
  \newcommand*\fsize{\dimexpr\f@size pt\relax}%
  \newcommand*\lineheight[1]{\fontsize{\fsize}{#1\fsize}\selectfont}%
  \ifx\svgwidth\undefined%
    \setlength{\unitlength}{391.81884384bp}%
    \ifx\svgscale\undefined%
      \relax%
    \else%
      \setlength{\unitlength}{\unitlength * \real{\svgscale}}%
    \fi%
  \else%
    \setlength{\unitlength}{\svgwidth}%
  \fi%
  \global\let\svgwidth\undefined%
  \global\let\svgscale\undefined%
  \makeatother%
  \begin{picture}(1,0.20594366)%
    \lineheight{1}%
    \setlength\tabcolsep{0pt}%
    \put(0,0){\includegraphics[width=\unitlength,page=1]{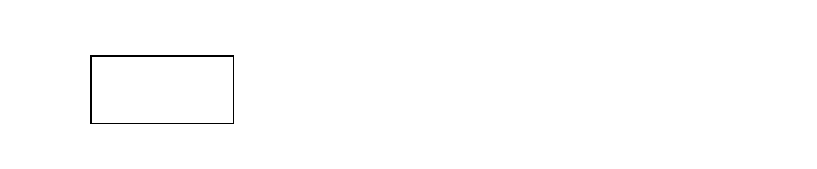}}%
    \put(0.18945716,0.09060661){\makebox(0,0)[lt]{\lineheight{1.25}\smash{\begin{tabular}[t]{l}$D$\end{tabular}}}}%
    \put(0,0){\includegraphics[width=\unitlength,page=2]{gauge-transformation-A.pdf}}%
    \put(0.40528314,0.1826748){\makebox(0,0)[lt]{\lineheight{1.25}\smash{\begin{tabular}[t]{l}$h$\end{tabular}}}}%
    \put(0,0){\includegraphics[width=\unitlength,page=3]{gauge-transformation-A.pdf}}%
    \put(0.72968021,0.09060661){\makebox(0,0)[lt]{\lineheight{1.25}\smash{\begin{tabular}[t]{l}$D'$\end{tabular}}}}%
    \put(0,0){\includegraphics[width=\unitlength,page=4]{gauge-transformation-A.pdf}}%
    \put(0.9441607,0.1826715){\makebox(0,0)[lt]{\lineheight{1.25}\smash{\begin{tabular}[t]{l}$h$\end{tabular}}}}%
    \put(0,0){\includegraphics[width=\unitlength,page=5]{gauge-transformation-A.pdf}}%
  \end{picture}%
\endgroup%
}
  \subcaptionbox{ \(D''\) is obtained from \(D\) by a type (B) gauge transformation by \(h\).}{
\begingroup%
  \makeatletter%
  \providecommand\color[2][]{%
    \errmessage{(Inkscape) Color is used for the text in Inkscape, but the package 'color.sty' is not loaded}%
    \renewcommand\color[2][]{}%
  }%
  \providecommand\transparent[1]{%
    \errmessage{(Inkscape) Transparency is used (non-zero) for the text in Inkscape, but the package 'transparent.sty' is not loaded}%
    \renewcommand\transparent[1]{}%
  }%
  \providecommand\rotatebox[2]{#2}%
  \newcommand*\fsize{\dimexpr\f@size pt\relax}%
  \newcommand*\lineheight[1]{\fontsize{\fsize}{#1\fsize}\selectfont}%
  \ifx\svgwidth\undefined%
    \setlength{\unitlength}{391.81884384bp}%
    \ifx\svgscale\undefined%
      \relax%
    \else%
      \setlength{\unitlength}{\unitlength * \real{\svgscale}}%
    \fi%
  \else%
    \setlength{\unitlength}{\svgwidth}%
  \fi%
  \global\let\svgwidth\undefined%
  \global\let\svgscale\undefined%
  \makeatother%
  \begin{picture}(1,0.20594366)%
    \lineheight{1}%
    \setlength\tabcolsep{0pt}%
    \put(0,0){\includegraphics[width=\unitlength,page=1]{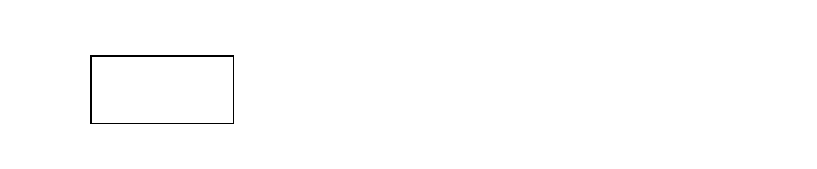}}%
    \put(0.18945716,0.09060661){\makebox(0,0)[lt]{\lineheight{1.25}\smash{\begin{tabular}[t]{l}$D$\end{tabular}}}}%
    \put(0,0){\includegraphics[width=\unitlength,page=2]{gauge-transformation-B.pdf}}%
    \put(0.40528314,0.1826748){\makebox(0,0)[lt]{\lineheight{1.25}\smash{\begin{tabular}[t]{l}$h$\end{tabular}}}}%
    \put(0,0){\includegraphics[width=\unitlength,page=3]{gauge-transformation-B.pdf}}%
    \put(0.72968021,0.09060661){\makebox(0,0)[lt]{\lineheight{1.25}\smash{\begin{tabular}[t]{l}$D''$\end{tabular}}}}%
    \put(0,0){\includegraphics[width=\unitlength,page=4]{gauge-transformation-B.pdf}}%
    \put(0.9441607,0.1826715){\makebox(0,0)[lt]{\lineheight{1.25}\smash{\begin{tabular}[t]{l}$h$\end{tabular}}}}%
    \put(0,0){\includegraphics[width=\unitlength,page=5]{gauge-transformation-B.pdf}}%
  \end{picture}%
\endgroup%
}
  \caption{
    Graphical description of gauge transformations.
    The thick strands represent bundles of incoming and outgoing segments (possibly empty) with arbitrary colorings.
  }
  \label{fig:gauge-transformation}
\end{figure}

The letters A and B are chosen because a type (A) transformation comes from pulling a strand colored by \(h\) across and \textbf{A}bove the diagram, while for type (B) we pull across and \textbf{B}elow.
Both are shown in \cref{fig:gauge-transformation}.
These are quite similar to the diagrammatic gauge transformations of \textcite[Section 3.3]{Blanchet2018}.
Combinatorially they come from Reidemeister moves for shadow colored diagrams.

Let \(i \mapsto \chi_i\) be the octahedral coloring associated to a shadow coloring.
We can immediately compute that under a type (A) transformation by \(h\) the new associated coloring is
\begin{equation}
  \label{eq:type-A-gauge}
  i \mapsto
  \tilde{\chi}_i
  =
  \left(
    \frac{
      \det( h^{-1} u_{\dnr{i}} , \evec{2} )
    }{
      \det( h^{-1} u_{\upr{i}} , \evec{2} )
    }
    ,
    -
    \frac{
      v_i h \evec{2}
    }{
      v_i h h^{-1} u_j
    }
    ,
    m_i
  \right)
  =
  \left(
    \frac{
      \det( u_{\dnr{i}} , h \evec{2} )
    }{
      \det( u_{\upr{i}} , h \evec{2} )
    }
    ,
    -
    \frac{
      v_i h \evec{2}
    }{
      v_i u_{\upr{i}}
    }
    ,
    m_i
  \right)
\end{equation}
while under type (B) it is
\begin{equation}
  \label{eq:type-B-gauge}
  i \mapsto
  \tilde{\chi}_i
  =
  \left(
    \frac{
      \det( h^{-1} u_{\dnr{i}} , \evec{2} )
    }{
      \det( h^{-1} u_{\upr{i}} , \evec{2} )
    }
    ,
    -
    \frac{
      v_i \evec{2}
    }{
      v_i  h^{-1} u_j 
    }
    ,
    m_i
  \right)
  =
  \left(
    \frac{
      \det( u_{\dnr{i}} , h \evec{2} )
    }{
      \det( u_{\upr{i}} , h \evec{2} )
    }
    ,
    -
    \frac{
      v_i \evec{2}
    }{
      v_i  h^{-1} u_{\upr{i}}
    }
    ,
    m_i
  \right)
\end{equation}

\begin{bigtheorem}
  \label{thm:gauge-equivalent-to-admissible}
  Let \(D\) be any tangle diagram and \(\rho : \pi(D) \to \slg\) any decorated representation.
  Then \(\rho\) is gauge-equivalent to an admissible representation, in the sense that there is a conjugate \(h^{-1} \rho h\) of \(\rho\) that gives an admissible decorated shadow coloring of \(D\).
\end{bigtheorem}

Another way to state this theorem is to consider the holonomy as a map from the space of octahedral colorings of \(D\) to the space of decorated \(\slg\) representations of \(D\).
This map is not surjective because there are inadmissible shadow colorings, but it is surjective up to gauge in the sense that the image intersects the orbit of any \(\rho\).

\begin{proof}
  Fix a shadow coloring \(i \mapsto (g_i, v_i), j \mapsto u_j\) of \(D\).
  We need to show that some gauge transformation of it is admissible.
  The coloring obtained by applying a type (A) gauge transformation by \(A\) then a type (B) by \(B^{-1}\) is admissible when
  \[
    \det(B u_j , A \evec{2} ), v_i A \evec{2}, v_i B u_{i^{\uparrow}} \ne 0
  \]
  for every region \(j\) and segment \(i\), where as usual the \(v_i\) are representatives of the eigenspaces \(L_i\).
  Thinking of the \(u_j\) and \(v_i\) as fixed we need to show that there are \(A, B \in \slg\) for which this condition holds.

  \(V = \slg \times \slg\) is an algebraic set in \(\mathbb{C}^{8}\).
  For each region \(j\) write
  \[
    X_1(j) = \set{(A,B) \in V | \det(A u_j , B \evec{2}) = 0 }
  \]
  and for each segment \(i\) write
  \[
    \begin{gathered}
      X_2(i) = \set{(A,B) \in V | v_i A \evec{2} = 0 }
    \\
    X_3(i) = \set{(A,B) \in V | v_i B u_{i^{\uparrow}} = 0 }
    .
    \end{gathered}
  \]
  These are all Zariki closed proper subsets of \(V\).
  Our goal is to prove that the union
  \[
    X =
    \bigcup_{\text{regions } j} X_1(j)
    \cup
    \bigcup_{\text{segments } i} X_2(i) \cup X_3(i)
  \]
  is not all of \(\slg \times \slg\), or equivalently that
  \[
    V \setminus X = 
    \bigcap_{\text{regions } j} U_1(j)
    \cap
    \bigcap_{\text{segments } i} U_2(i) \cap U_3(i)
  \]
  is nonempty, where \(U_l(k) = V \setminus X_l(k)\).
  Each complement \(U_l(k)\) is Zariski open and dense.
  A finite intersection of such sets is also Zariski open and dense, and in particular is nonempty.
\end{proof}

\section{Shape parameters and the potential function}
\label{sec:potential}

One way to subdivide the octahedral decomposition is to split the octahedron at each crossing into four tetrahedra as in \cref{fig:shape-parameters-diagram}, with one for each region touching the crossing.
We can geometrize the tetrahedra by placing their vertices in the boundary at infinity of hyperbolic space, i.e. \(\psp\).
Such geometric ideal tetrahedra are determined up to congruence by the cross-ratio of their vertices, usually called the \defemph{shape parameter}.
In this case these parameters are given by \cref{eq:b-shapes}.%
\note{
  For details see \cite{McPhailSnyder2022}.
  Note that we are using a different convention on negatively oriented tetrahedra.
}
When a shape parameter is \(0\), \(1\), or \(\infty\) it represents a geometrically degenerate tetrahedron whose vertices coincide.
\cref{thm:gauge-equivalent-to-admissible} says that for a given representation we can always avoid tetrahedra with shape \(0\) or \(\infty\).
In this section we study when tetrahedra with shape \(1\) occur.

\begin{figure}
  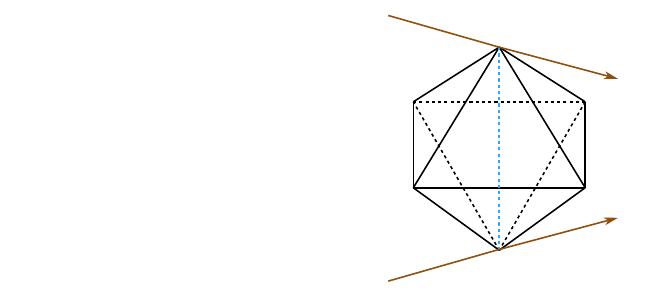
  \caption{%
    One can divide the octahedron at a crossing into four tetrahedra, one for each region touching the crossing.
    The shape parameters are then given as ratios of the \(b\) and \(m\)-coordinates, as in \cref{eq:b-shapes}.
  }
  \label{fig:shape-parameters-diagram}
\end{figure}

\begin{definition}
  An octahedral coloring of a crossing is \defemph{pinched}%
  \note{%
    This term is due to \textcite{Kim2016}.
    At a pinched crossing the ideal points of the tetrahedra are ``pinched'' together.
  }
  if any of the shape parameters
  \begin{equation}
    \label{eq:b-shapes}
    z_N
    =
    \frac{
      b_{2'}  
      }{
      b_1   
    }
    ,
    z_W
    =
    \frac{
      b_2
      }{
      m_1 b_1
    }
    ,
    z_S
    =
    \frac{
      m_2 b_2 
      }{
      m_1 b_{1'}
    }
    ,
    z_E
    =
    \frac{
      m_2 b_{2'}
      }{
      b_{1'}
    }
  \end{equation}
  are \(1\).
  (These quantities can never be \(0\) or \(\infty\) as part of the definition of octahedral coloring.)
  An octahedral coloring of a diagram is pinched if any of its crossings are.
\end{definition}

We can use \cref{def:associated-octahedral} to directly find the shape parameters of the octahedral coloring associated to a decorated representation.
As a corollary we see that being pinched is a gauge-invariant property.

\begin{theorem}
  \label{thm:shape-formula}
  Let \(D\) be a tangle diagram with an admissible decorated shadow \(\slg\) coloring \(i \mapsto (g_i, L_i), j \mapsto u_j\) and let \(c\) be a crossing of \(D\).
  \begin{thmenum}
    \item
    \label{thm:shape-formula:formula}
      The shape parameters of the crossing  are given by
      \begin{align}
        \label{eq:zN-formula}
        z_N
      &=
      \frac{
        v_1 u_N
        }{
        v_1 \evec{2}
      }
      \frac{
        v_{2'} \evec{2}
        }{
        v_{2'} u_N
      }
      =
      \frac{
        h(v_{1} \upf{u_N})
        }{
        h(v_{2'} \upf{u_N})
      }
      \\
      \label{eq:zW-formula}
      z_W
      &=
      \frac{
        v_{1} u_W
        }{
        v_{1} \evec{2}
      }
      \frac{
        v_{2} \evec{2}
        }{
        v_{2} u_W
      }
      =
      \frac{
        h(v_{1} \upf{u_W})
        }{
        h(v_{2} \upf{u_W})
      }
      \\
      \label{eq:zS-formula}
      z_S
      &=
      \frac{
        v_{1'} u_S
        }{
        v_{1'} \evec{2}
      }
      \frac{
        v_{2'} \evec{2}
        }{
        v_{2'} u_S
      }
      =
      \frac{
        h(v_{1'} \upf{u_S})
        }{
        h(v_{2'} \upf{u_S})
      }
      \\
      \label{eq:zE-formula}
      z_E
      &=
      \frac{
        v_{1'} u_E
        }{
        v_{1'} \evec{2}
      }
      \frac{
        v_{2'} \evec{2}
        }{
        v_{2'} u_E
      }
      =
      \frac{
        h(v_{1'} \upf{u_E})
        }{
        h(v_{2'} \upf{u_E})
      }
    \end{align}
    in terms of the \defemph{Hopf map}
    \[
      h\left(
        \begin{bmatrix}
          v^1 & v^2
        \end{bmatrix}
      \right)
      \defeq \frac{v^1}{v^2}.
    \]
    Here the regions and segments are labeled as in \cref{fig:crossing-types} and the \(v_i\) are nonzero elements of the \(L_i\) (i.e.\@ representative eigenvectors).
  \item
    \label{thm:shape-formula:pinched}
    The crossing is pinched if and only if \(L_1 = L_2\).
    Note that this condition does not depend on the choice of shadow coloring and is invariant under gauge transformations.
      \qedhere
\end{thmenum}
\end{theorem}

\begin{proof}
  \ref{thm:shape-formula:formula}
  \Cref{eq:zN-formula} follows directly from the expressions for \(b_1\) and \(b_{2'}\) in \cref{def:associated-octahedral} and
  \[
    v_1 \upf{u_N} = \begin{bmatrix}
      v_1 u_N & v_1 \evec{2}
    \end{bmatrix}
    .
  \]
  The same method gives
  \[
    z_W = 
    m_1^{-1}
    \frac{
      v_{1} u_N
      }{
      v_{1} \evec{2}
    }
    \frac{
      v_{2} \evec{2}
      }{
      v_{2} u_W
    }
    =
    m_1^{-1}
    \frac{
      h(v_{1} \upf{u_N})
      }{
      h(v_{2} \upf{u_W})
    }
    .
  \]
  To obtain \cref{eq:zW-formula}, write
  \[
    m_1^{-1} v_1 u_N
    =
    v_1 g_1 u_N
    =
    v_1 u_W
  \]
  as \(u_W = g_1 u_N\) by the definition of shadow coloring.
  \Cref{eq:zS-formula,eq:zE-formula} follow from the same argument.

  \ref{thm:shape-formula:pinched}
  One can use \cref{eq:b-transf-positive,eq:b-transf-negative} to show that \(z_W = 1\) if and only if the crossing is pinched.
  \(\upf{u_W}\) is invertible, so \(v_1 \upf{u_W}\) and \(v_2 \upf{u_W}\) are dependent iff \(v_1\) and \(v_2\) are.
\end{proof}

The following corollary is useful when defining state integrals for the quantum invariants constructed in \cite{McPhailSnyderVolume}.

\begin{corollary}
  If \(\rho : \pi(D) \to \slg\) is a non-pinched representation, then it is gauge-equivalent to a shadow coloring for which no shape parameter lies on the unit circle.
\end{corollary}

\begin{proof}
  The argument is similar to the proof of \cref{thm:gauge-equivalent-to-admissible}.
  The shape parameter for a corner by segments \(i\) and \(k\) near region \(j\) is
  \[
    z_{ik}^{j} =
    \frac{
      v_i u_j
      }{
      v_k u_j
    }
    \frac{
      v_k \evec{2}
      }{
      v_i \evec{2}
    }
  \]
  By \cref{eq:type-A-gauge} after applying a type (A) gauge transformation by \(A \in \slg\) the shape parameter
  \[
    z_{ik}^{j} =
    \frac{
      v_i u_j
      }{
      v_k u_j
    }
    \frac{
      v_k A \evec{2}
      }{
      v_i A \evec{2}
    }
    =
    C_{ik}^{j}
    \frac{
      v_k x
      }{
      v_j x
    }
  \]
  where \(C_{ik}^j = v_i u_u / v_k u_j\) and \(x = A \evec{2}\).
  This is homogeneous in \(x\) so we can think of it as a function defined on \(\psp\).
  Let
  \[
    X_{i k}^j = 
    \set{ [x]  \given |C_{ik}^{j} v_j x|^2 = |v_k x|^2 } \subset \psp
  \]
  be the set of points for which this particular shape parameter lies on the unit circle.
  Our claim is that the union \(X = \bigcup X_{ik}^{j}\) is not all of \(\psp\).
  Because this is a finite union it is enough to show that each \(X_{ik}^{j}\) is a circle in the Riemann sphere.
  
  Because no crossing is pinched \(v_i\) and \(v_k\) are independent, so there is a fractional linear transformation \(B \in \slg\) with \(v_i B = \evec{1}^{\transp}\) and  \(v_k B = \evec{2}^{\transp}\).
  If we change variables to \(y = (y_0, y_1) = B x\) our equation becomes
   \[
    |C_{i,k}^{j}|^2 |y_0|^2 = |y_1|^2
  \]
  Viewed as an equation for \(h(y) = y_0/y_1 \) in the Riemann sphere this describes a circle.
  Because fractional linear transformations preserve circles we see that \(X_{ik}^{j}\) is a circle as well.
\end{proof} 

To explain the connection to the Volume Conjecture we introduce a version of the Neumann-Zagier--Yokota potential function of the diagram defined using the dilogarithm
\[
  \dil(z)
  \defeq
  \int_{0}^{z}
  \frac{
    - \log(1-t)
    }{
    t
  }
  dt
  .
\]
\(\dil\) has a branch point at \(1\), where it is continuous but not differentiable.
It can be used to compute the volumes \cite{Zagier2007} and Chern-Simons invariants \cite{Neumann2004} of hyperbolic \(3\)-manifolds.
Set
\begin{equation}
  \logdil(\zeta)
  \defeq
  \frac{
    \dil(e^{2\pi \ii \zeta})
    }{
    2\pi \ii
  }
  .
\end{equation}
Let \(D\) be a link with \(\ell\) components.
Fix a choice \(\mu = (\mu_1, \dots, \mu_\ell) \) of complex number for each component of \(\ell\) and introduce a variable \(\beta_1, \dots, \beta_n\) for each segment of \(D\).
We think of the \(\beta_i\) and \(\mu_j\) as logarithms of the \(b\) and \(m\) coordinates of some octahedral coloring.
For each crossing \(c\) (as in \cref{fig:shape-parameters-diagram}) of \(D\) set
\[
  \potl[c]{} =
  \logdil(\beta_{2'} - \beta_{1})
  -\logdil(\beta_2 - \beta_1 - \mu_1)
  +\logdil(\beta_{2} - \beta_{1'} + \mu_2 - \mu_1)
  -\logdil(\beta_{2'} - \beta_{1'} + \mu_2)
  .
\]
The arguments are logarithms of the shape parameters \eqref{eq:b-shapes}.

\begin{definition}
  The \defemph{potential function} of \(D\) with respect to \(\mu\) is
  \begin{equation}
    \label{eq:potential-function}
    \potl[D, \mu]{\beta_1, \dots, \beta_n}
    =     
    \sum_{\text{crossings } c}
    \epsilon(c) \Phi_c
  \end{equation}
  where \(\epsilon(c) = +1\) for positive crossings and \(-1\) for negative crossings.
  Because \(\dil\) has a branch point at \(1\) this function is analytic when the arguments of the functions \(\logdil\) (i.e.\@ the logarithms of the shape parameters) are not in \(2\pi \ii \mathbb{Z}\).
\end{definition}

\begin{theorem}
  \label{thm:critical-points-are-reps}
  Let \(\mathcal{C}_{D, \mu}\) be the set of generalized critical points of \(\potl[D, \mu]{}\), i.e.\@ points \(\beta = (\beta_1, \dots, \beta_n)\) where
  \[
      \frac{
        \partial \potl[D, \mu]{}
      }{
        \partial \beta_i
      }
      \equiv
      0
      \pmod{2\pi \ii}
      \text{ for }
      i = 1, \dots, n.
  \]
  Then there are numbers \(a_i\) so that
  \[
    i \mapsto (a_i, \exp(2\pi \ii \beta_i), \exp(2\pi \ii \mu_{k(i)}) )
  \]
  is an octahedral coloring of \(D\), where \(\mu_{k(i)}\) is the meridian log-parameter associated with the component \(k(i)\) of segment \(i\).
\end{theorem}

\begin{proof}
  By \cite[Theorem 6.3]{McPhailSnyder2022} non-pinched colorings are in bijection with solutions of the \defemph{segment equations} of the diagram;
  these are a set of rational equations in the \(b\)-variables, with one for each segment.
  A solution of the segment equations determines an \(a\)-variable for each segment yielding an octahedral coloring, and it exists if and only if the gluing equations of the four-term decomposition have nondegenerate solutions.
  Because 
  \[
    \frac{
      d
      }{
      d\zeta
    }
    \logdil(\zeta)
    =
    -
    \log(1 - e^{2\pi \ii \zeta})
  \]
  we can work out that
  \[
    \exp
    \frac{
      \partial \potl[D, \mu]{}
      }{
      \partial \beta_i
    }
    =
    1
  \]
  is the segment equation of segment \(i\).
\end{proof}

One can easily handle pinched crossings by working directly with octahedral colorings or by using shadow \(\slg\)-colorings and \cref{thm:holonomies-match}.
However, the smoothness of \(\potl[D, \mu]{}\) is still significant for the Volume Conjecture.
Most proofs follow \textcite{Yokota2000} by showing that the large \(N\) asymptotics of the Kashaev invariant are given by a contour integral of the form
\[
  \int_{\Gamma} e^{N \potl[D, \mu]{\beta_1, \dots, \beta_n}} \,d\beta_1 \cdots d\beta_n
\]
as \(N \to \infty\).
The saddle-point method for determining the asymptotics of these integrals works best near analytic critical points of \(\potl[D,\mu]{}\).
More significantly the proof usually requires moving the contour \(\Gamma\) through the parameter space which requires \(\potl[D, \mu]{}\) to be complex-analytic.

As part of this argument it is important to know whether a representation \(\rho\) occurs as the holonomy of an octahedral coloring from an \emph{analytic} critical point.
Typically we only care about \(\rho\) up to conjugacy, which motivates the following definition:

\begin{definition}
  Let \(\rho : \pi(L) \to \slg\) be a representation of a link group and \(D\) a diagram of \(L\).
  We say that \(\rho\) is \defemph{\(D\)-smooth} if it is conjugate to the holonomy of a non-pinched octahedral coloring \(\chi\) of \(D\), i.e.\@ to the holonomy of a generalized critical point.
\end{definition}

\begin{corollary}
  \label{thm:detection-condition}
  Let \(D\) be any link diagram and \(\rho : \pi(D) \to \slg\) a decorated representation.
  Then \(\rho\) is \(D\)-smooth if and only if for each crossing of \(D\) the lines \(L_1\) and \(L_2\) are distinct.
\end{corollary}

\begin{proof}
  \cref{thm:shape-formula}\ref{thm:shape-formula:pinched}.
\end{proof}

In other words, \(\rho\) corresponds to a pinched octahedral coloring if and only if at some crossing the images of the Wirtinger generators have the same distinguished fixed point when acting on \(\psp\) by fractional linear transformations.
If we restrict our attention to hyperbolic knots then we can give a purely group-theoretic characterization.
Recall that the holonomy of the complete hyperbolic structure is a faithful, discrete representation \(\rho : \pi(K) \to \pslg\).
It is boundary-parabolic (the meridians have eigenvalues \(\pm 1\)) and always lifts to \(\slg\).
Because its meridians are parabolic it has a unique decoration.

\begin{definition}
  A tangle diagram \(D\) is \defemph{arc-faithful} if at each crossing the Wirtinger generators of the over and under arcs are distinct elements of \(\pi(T)\).%
  \note{%
    Since \(w_1 = w_2\) if and only if \(w_1 = w_1^{-1} w_2 w_1\) it does not matter which under arc is used.
  }
\end{definition}

\begin{bigtheorem}
  \label{thm:arc-faithful}
  Let \(D\) be a diagram of a hyperbolic knot.
  The holonomy \(\rho\) of its complete hyperbolic structure is \(D\)-smooth if and only if \(D\) is arc-faithful.
\end{bigtheorem}
\begin{proof}
  The only if is obvious: if two adjacent arcs have the same Wirtinger generators then their images always share a fixed point.

  Conversely, suppose \(\rho\) corresponds to a pinched octahedral coloring of \(D\) so there are Wirtinger generators \(w_1\) and \(w_2\) for which \(\rho(w_1)\) and \(\rho(w_2)\) have the same fixed point.
  As \(\rho(w_1)\) and \(\rho(w_2)\) are parabolic this implies they commute, hence \(w_1\) and \(w_2\) commute as \(\rho\) is faithful.
  By the classification of finitely generated abelian subgroups of \(3\)-manifold groups \cite[Theorem 9.13]{Hempel1976} two commuting infinite-order elements generate a subgroup isomorphic to \(\mathbb{Z} \) or \(\mathbb{Z} \times \mathbb{Z} \).
  As \(w_1\) and \(w_2\) map to the same element of the first homology they must generate a \(\mathbb{Z}\) subgroup.
  They are primitive and map to the same homology class so they must be equal.
\end{proof}

The existence of analytic critical points of \(\potl[D, \mu]{}\) has previously been studied, and we can express our results in this context.
An edge of an ideal triangulation of a knot complement is (homtopically) \defemph{peripherial} if it is homotopic to a curve in the boundary of a regular neighborhood of the knot.

\begin{corollary}
  \label{thm:non-peripheral}
  Let \(D\) be a diagram of a hyperbolic knot.
  \begin{thmenum}
  \item
    \label{thm:non-peripheral:equiv}
    The edges of the four-term octahedral decomposition associated to \(D\) are homotopically non-peripheral if and only if \(D\) is arc-faithful.
  \item
    \label{thm:non-peripheral:alternating}
    If \(D\) is reduced and alternating then it is arc-faithful.
    \qedhere
  \end{thmenum}
\end{corollary}
\begin{proof}
  \ref{thm:non-peripheral:equiv}
  If \(D\) is arc-faithful then there is a non-pinched octahedral coloring \(\chi\) with holonomy the complete hyperbolic structure.
  From \(\chi\) we can obtain a nondegenerate solution of the gluing equations of the octahedral decomposition, and by \cite[Lemma 3.5]{Dunfield2012} the existence of such a solution implies the edges of the octahedral decomposition are homotopically non-peripheral.
  
  Conversely, suppose the edges of the octahedral decomposition are homotopically non-peripheral.
  Then by the paragraph before \cite[Remark 3.4]{Dunfield2012} the complete hyperbolic structure occurs as a non-degenerate solution of the gluing equations, hence as a non-pinched octahedral coloring.

  \ref{thm:non-peripheral:alternating}
  \textcite{Garoufalidis2016} and \textcite{Sakuma2018} independently showed that the edges of the four-term decomposition of a reduced alternating diagram are non-peripheral.
  Part \ref{thm:non-peripheral:equiv} implies such a diagram is arc-faithful.
\end{proof}

It is easy to find diagrams that are not arc-faithful.
However, by avoiding these obvious bad configurations it also seems easy to find arc-faithful diagrams.

\begin{example}
  Any diagram with a kink
  \begin{center}
\begingroup%
  \makeatletter%
  \providecommand\color[2][]{%
    \errmessage{(Inkscape) Color is used for the text in Inkscape, but the package 'color.sty' is not loaded}%
    \renewcommand\color[2][]{}%
  }%
  \providecommand\transparent[1]{%
    \errmessage{(Inkscape) Transparency is used (non-zero) for the text in Inkscape, but the package 'transparent.sty' is not loaded}%
    \renewcommand\transparent[1]{}%
  }%
  \providecommand\rotatebox[2]{#2}%
  \newcommand*\fsize{\dimexpr\f@size pt\relax}%
  \newcommand*\lineheight[1]{\fontsize{\fsize}{#1\fsize}\selectfont}%
  \ifx\svgwidth\undefined%
    \setlength{\unitlength}{55.43417931bp}%
    \ifx\svgscale\undefined%
      \relax%
    \else%
      \setlength{\unitlength}{\unitlength * \real{\svgscale}}%
    \fi%
  \else%
    \setlength{\unitlength}{\svgwidth}%
  \fi%
  \global\let\svgwidth\undefined%
  \global\let\svgscale\undefined%
  \makeatother%
  \begin{picture}(1,0.81664445)%
    \lineheight{1}%
    \setlength\tabcolsep{0pt}%
    \put(0,0){\includegraphics[width=\unitlength,page=1]{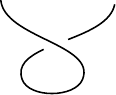}}%
  \end{picture}%
\endgroup%

  \end{center}
  is never-arc faithful.
  More generally any piece of a diagram that looks like
  \begin{center}
\begingroup%
  \makeatletter%
  \providecommand\color[2][]{%
    \errmessage{(Inkscape) Color is used for the text in Inkscape, but the package 'color.sty' is not loaded}%
    \renewcommand\color[2][]{}%
  }%
  \providecommand\transparent[1]{%
    \errmessage{(Inkscape) Transparency is used (non-zero) for the text in Inkscape, but the package 'transparent.sty' is not loaded}%
    \renewcommand\transparent[1]{}%
  }%
  \providecommand\rotatebox[2]{#2}%
  \newcommand*\fsize{\dimexpr\f@size pt\relax}%
  \newcommand*\lineheight[1]{\fontsize{\fsize}{#1\fsize}\selectfont}%
  \ifx\svgwidth\undefined%
    \setlength{\unitlength}{55.43913174bp}%
    \ifx\svgscale\undefined%
      \relax%
    \else%
      \setlength{\unitlength}{\unitlength * \real{\svgscale}}%
    \fi%
  \else%
    \setlength{\unitlength}{\svgwidth}%
  \fi%
  \global\let\svgwidth\undefined%
  \global\let\svgscale\undefined%
  \makeatother%
  \begin{picture}(1,0.8165715)%
    \lineheight{1}%
    \setlength\tabcolsep{0pt}%
    \put(0,0){\includegraphics[width=\unitlength,page=1]{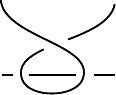}}%
  \end{picture}%
\endgroup%

  \end{center}
  will prevent arc-faithfulness.
\end{example}

According to \citeauthor{DiagramSite}'s database \cite{DiagramSite} of boundary-parabolic representations every hyperbolic knot with at most \(12\) crossings has an arc-faithful diagram.
We conjecture this holds for all hyperbolic knots, and as discussed above this seems relevant to a general proof of the Volume Conjecture.

\begin{conjecture}
  Every hyperbolic knot has an arc-faithful diagram.
\end{conjecture}

\printbibliography[heading=bibintoc,title={Bibliography}]

\end{document}